\documentclass{amsart}
\usepackage{amsmath,amsthm,amssymb,enumerate,epsfig,color}
\def\<{\langle}
\def\>{\rangle}

\def\ol{\overline}

\def\C{{\mathbb C}}
\def\N{{\mathbb N}}
\def\Z{{\mathbb Z}}
\def\R{{\mathbb R}}

\def\co{\colon \thinspace}

 \def\wtil{\widetilde}
\renewcommand{\phi}{\varphi}
\newtheorem{lemma}{Lemma}[section]
\newtheorem{theorem}[lemma]{Theorem}
\newtheorem{prop}[lemma]{Proposition}

\newtheorem{lem}[lemma]{Lemma}
\newtheorem{sublem}[lemma]{Sublemma}
\newtheorem{conj}[lemma]{Conjecture}

\newtheorem{example}[lemma]{Example}
\newtheorem{question}[lemma]{Question}

\begin{document}

\title[How to read the length of a braid from its curve diagram]
{How to read the length of a braid\\
from its curve diagram}
\author{Bert Wiest}
\address{IRMAR, UMR 6625 du CNRS, Universit\'e de Rennes 1, France}
\email{\tt bertold.wiest@univ-rennes1.fr}
\begin{abstract}
We prove that the Garside length a braid is equal
to a winding-number type invariant of the curve diagram of the braid.
\end{abstract}
\maketitle


\section{Introduction}

Is it possible to read the length of a braid~$\beta$ from the curve 
diagram of~$\beta$? The words used in this question admit different
interpretations, but in this paper we shall show that, if the word 
``length'' is interpreted in the sense of Garside, then the answer 
is affirmative:
the Garside length of a braid is equal to a winding-number type invariant
of the curve diagram which can be read from the diagram by a
simple-minded procedure.

The easiest answer one might have hoped for is that the 
length of a braid is proportional to the number of intersections of the 
curve diagram of the braid with the real line.
This answer, however, is false -- see for instance \cite{RafiDistFormula} 
and~\cite{DynnWie}. 
In these papers a relation was established between a certain
\emph{distorted} word length and the above-mentioned intersection
number, which is in turn related to the distance in Teichm\"uller space 
between a base point and its image under the braid action.

Let us recall some basic definitions and establish some notation.
We recall that Garside introduced in~\cite{Garside} a certain
set of generators for the braid group, called the Garside generators 
(or ``divisors of~$\Delta$'' or ``positive permutation braids''). 
The length of an element of the braid group with respect to this
generating set is called its \emph{Garside length}. 
Still according to Garside, every braid~$\beta$ can be written 
in a canonical way as a product of Garside generators and their 
inverses (see e.g.~\cite{ElrifaiMorton,
DehGars,GebGMCyclicSliding}).
From this canonical form one can read off 
two integer numbers $\inf(\beta)$ and $\sup(\beta)$, which are the
the maximal and minimal integers, respectively, satisfying
$$
\Delta^{\inf(\beta)}\preccurlyeq \beta \preccurlyeq \Delta^{\sup(\beta)}
$$
(Here the symbol~$\preccurlyeq$ denotes the subword partial ordering:
$\beta_1\preccurlyeq\beta_2$ means that there exists a product~$w$ of
Garside generators such that $\beta_1\cdot w=\beta_2$.
Moreover, $\Delta$ denotes Garside's half twist braid $\sigma_1 \cdot
(\sigma_2 \sigma_1)\cdot\ldots\cdot(\sigma_{n-1}\ldots\sigma_1)$.) 
One can show that Garside's canonical representative of~$\beta$ 
realizes the Garside length of~$\beta$, which is thus equal to \ 
$\max(\sup(\beta),0)-\min(\inf(\beta),0)$.

We denote~$D_n$ the $n$ times punctured disk, i.e.~the disk
$D^2=\{c\in\C\ |\ |c|\leqslant 1\}$, equipped with~$n$
punctures which are regularly spaced in the interior of
the interval $[-1,1]$.
We define~$E$ to be the diagram in $D_n$ consisting of the
real (horizontal) line segment between the leftmost and the rightmost
puncture. We define~$\ol E$ to be the diagram consisting of the real line
segment between the point~$-1$ (the leftmost point of $\partial D^2$)
and the rightmost puncture. A \emph{curve diagram of a braid~$\beta$}
is the image of $\ol E$ under a diffeomorphism of~$D_n$ representing 
the braid~$\beta$ (c.f.\cite{FGRRW}). Throughout this paper, braids
act on the right. We shall call a curve diagram \emph{reduced} if it
has the minimal possible number of intersections with the horizontal
line, and also the minimal possible number of vertical tangencies
in its diffeotopy class, and if none of the punctures lies in a
point of vertical tangency.


\section{Winding number labellings and the main result}

For a braid~$\beta$ we are going to denote~$\ol D_\beta$ the curve
diagram of~$\beta$, and~$D_\beta$ the image of~$E$ under~$\beta$
-- so $D_\beta$ is obtained from $\ol D_\beta$ simply by removing
one arc. 
We shall label each segment of $\ol D_\beta$
between two subsequent vertical tangencies by an integer number, in the
following way: the first segment (which starts on~$\partial D^2$) is
labelled~$0$, and if the label of the $i$th segment is $k$, and if the
transition from the $i$th to the $i+1$st segment is via a \emph{right}
curve, then the $i+1$st segment is labelled $k+1$. If, on the other hand,
the transition is via a \emph{left} curve, then the label of the
$i+1$st segment is $k-1$. See Figure~\ref{F:LabelEx} for an example
of this labelling. We shall call this labelling the \emph{winding
number labelling} of the curve diagram.

A more rigorous definition of this labelling is as follows. If 
$\alpha\co I \to D^2$ is a smooth parametrization of the curve diagram
$\ol D_\beta$, defined on the unit interval $I=[0,1]$, and such that
$\alpha(0)=-1$, then
we define the \emph{tangent direction function} 
$T_\alpha\co I\to \R / 2\Z$ as the angle of the tangent direction
of~$\alpha$ against the horizontal, divided by~$-\pi$. In particular,
if the arc goes straight to the right in $\alpha(t)$, then
$T_\alpha(t)=0+2\Z$. If it goes straight down, then 
$T_\alpha(t)=\frac{1}{2}+2\Z$; and if it goes to the left, then
$T_\alpha(t)=1+2\Z$. 
Now we have a unique lifting of the function~$T_\alpha$ to a function 
$\wtil{T}_\alpha\co I\to \R$ with $\wtil{T}_\alpha(0)=0$.
Finally, if $r\co \R\to\Z$ denotes the rounding function, which
sends every real number to the nearest integer (rounding \emph{down}
$n+\frac{1}{2}$), then we define the function
$$
\tau_\alpha\co I\to \R, \ t\mapsto r\circ\wtil{T}_\alpha(t)
$$
which one might call the rounded lifted tangent direction function.

Now the winding number labelling can be redefined as follows: a point~$x$
of~$D_\beta$ with non-vertical tangent direction is labelled by the
integer $\tau_\alpha(t)$, where~$t$ in~$I$ is such that $\alpha(t)=x$.

\begin{figure}[htb]
\centerline{\input{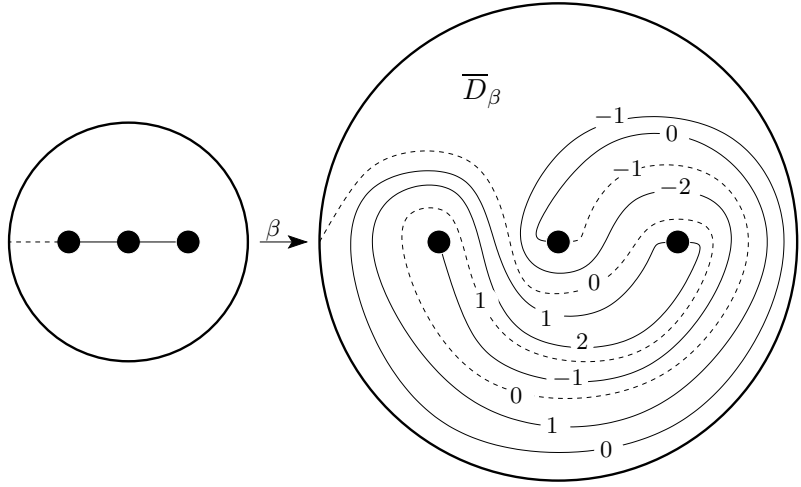}}
\caption{The curve diagram (with the first line drawn dashed) of the
braid $\beta=(\sigma_1\sigma_2^{-1})^2$. The labels of the solid arcs vary
between $-2$ and $2$. The Garside normal form of~$\beta$ is 
$\sigma_2^{-1}\sigma^{-1}_{1}\cdot \sigma^{-1}_1\cdot 
\sigma_{2}\cdot \sigma_{2}\sigma_{1}$, so its infimum is~$-2$ and
its supremum is~$2$. Also note that there is only one arc carrying
the maximal label~``$2$'' and only one arc with the minimal
label~``$-2$''.}
\label{F:LabelEx}
\end{figure}

We shall denote the largest and smallest labels occurring in~$D_\beta$ 
by $LL(\beta)$ and $SL(\beta)$. Notice that we are talking about the 
restricted diagram $D_\beta$, not the full diagram $\ol D_\beta$. 
Nevertheless, in order
to actually calculate the labels of the restricted diagram, one has to 
start by calculating the labels of the first arc (the one that starts at 
$-1$ on $\partial D^2$).

\begin{theorem}\label{T:label=length} For any braid $\beta$,
the largest label occurring in the diagram~$D_\beta$ is equal to 
$\sup(\beta)$ and the smallest label occurring is
equal to $\inf(\beta)$. In particular, the Garside length of
the braid~$\beta$ is $\max(LL(\beta),0)-\min(SL(\beta),0)$.
\end{theorem}

\begin{proof}[Proof of Theorem~\ref{T:label=length}] 
The proof comes down to the following lemma:

\begin{lem}\label{L:label=length}
If $\beta$ is a positive braid, then the largest label occurring in the
diagram~$D_\beta$ is equal to the Garside length of~$\beta$.
\end{lem}

Let us first see why Lemma~\ref{L:label=length} implies 
Theorem~\ref{T:label=length}:  let us assume for the moment that
Lemma~\ref{L:label=length} is true, and try to deduce 
Theorem~\ref{T:label=length}. The crucial observation is that
multiplying~$\beta$ by $\Delta^k$  
increases all four numbers (the sup, the inf, the maximal label 
of~$D_\beta$ and the minimal label of~$D_\beta$) by~$k$.
We obtain
\begin{eqnarray*}
LL(\beta) & = & LL(\beta \Delta^{-\inf(\beta)})+\inf(\beta)\\
 & \stackrel{*}{=} & 
          \mathrm{length}_{\mathrm{Gars}}(\beta \Delta^{-\inf(\beta)}) 
          + \inf(\beta)\\
 & = & \sup(\beta \Delta^{-\inf(\beta)})+\inf(\beta) = \sup(\beta) 
\end{eqnarray*}
where $\stackrel{*}{=}$ denotes that the equality follows from 
Lemma~\ref{L:label=length}. Now we make the following

{\bf Claim } If~$\beta$ is a negative braid, then the smallest label 
occurring in~$D_\beta$, multiplied by~$-1$, is equal to the Garside length 
of~$\beta$.

In order to prove this claim, we consider the image~$\ol\beta$ of~$\beta$
under the homomorphism which replaces each crossing $\sigma_i^{\pm 1}$ 
by its negative crossing $\sigma_i^{\mp 1}$. Its curve 
diagram~$D_{\ol\beta}$ is just the mirror image, with 
respect to the horizontal line, of~$D_\beta$. We observe that
$$
\inf(\beta)=-\sup(\ol\beta), \ \ \sup(\beta)=-\inf(\ol\beta), \ \ 
LL(\beta)=-SL(\ol\beta), \ \ SL(\beta)=-LL(\ol\beta).
$$
The claim now follows from Lemma~\ref{L:label=length}.

Now we have for an arbitrary braid~$\beta$ that
\begin{eqnarray*}
SL(\beta) & = & SL(\beta \Delta^{-\sup(\beta)})+\sup(\beta)\\
 & \stackrel{*}{=} & 
            -\mathrm{length}_{\mathrm{Gars}}
            (\beta \Delta^{-\sup(\beta)}) + \sup(\beta)\\
 & = & \inf(\beta \Delta^{-\sup(\beta)})+\sup(\beta) = \inf(\beta) 
\end{eqnarray*}
where $\stackrel{*}{=}$ denotes that the equality follows from the Claim.
This completes the proof of Theorem~\ref{T:label=length}, assuming
Lemma~\ref{L:label=length}.
\end{proof}

\begin{proof}[Proof of Lemma~\ref{L:label=length}]
First we shall prove that for a positive braid~$\beta$ we have 
$LL(\beta)\leqslant \mathrm{length}_{\mathrm{Gars}}(\beta)$. 
By induction, this
is equivalent to proving the following: if~$\beta$ is a positive
braid and~$\beta_+$ is a divisor of~$\Delta$, then 
$LL(\beta\cdot\beta_+)\leqslant LL(\beta)+1$.

The action on the disk of any braid~$\beta_+$ which is a divisor 
of~$\Delta$ can be realized by the following
dance of the punctures. Initially the punctures are lined up on the
real line. As a first step, perform a clockwise rotation so that 
the punctures are lined up on the imaginary axis. As a second step, 
move each of the punctures horizontally, until no more puncture lies 
precisely above any other one. In a third step, by vertical movements, 
bring all the punctures back to the horizontal line. 

Let us suppose that $\alpha\co I\to D^2$ is a parametrization of
the curve diagram~$\ol D_{\beta}$, and let us look at a segment 
$I'\subseteq I$ parametrizing a segment between two successive 
points of vertical tangency. The function $\tau_\alpha|_{I'}$
is constant, we shall denote its value by~$k$; this means that on the 
interval~$I'$, the function $\wtil T_\alpha$ takes values in the
interval $]k-\frac{1}{2},k+\frac{1}{2}]$. Now, after the first step
of the puncture dance (the rotation), we have a deformed curve diagram
parametrized by a function $\alpha'\co I\to D^2$, and we observe that
on the interval~$I'$, the function $\wtil T_{\alpha'}$ takes values in 
the interval $]k,k+1]$. The horizontal movement of the punctures
deforms the arc~$\alpha'$ into an arc~$\alpha''$, but during this
deformation no horizontal tangencies are created or destroyed;
therefore the values of $\wtil T_{\alpha''}|_{I'}$ still lie in the 
interval $]k,k+1]$. The third step (squashing the punctures
back to the real line) changes tangent directions by at most a
quarter turn, so the new function $\wtil T_{\alpha'''}$ takes
values in the interval $]k-\frac{1}{2},k+\frac{3}{2}]$ on~$I'$.
Thus on the interval~$I'$, we have $\tau_{\alpha'''}\in\{k,k+1\}$,
so the labels have increased by at most one under the action of~$\beta_+$.
This completes the proof that $LL(\beta)\leqslant 
\mathrm{length}_{\mathrm{Gars}}(\beta)$.

In order to prove the converse inequality, we shall prove that for any
braid~$\beta$ with $SL(\beta)\geqslant 0$, there is a braid~$\beta_-$ 
which is the inverse of a simple braid 
such that $LL(\beta\cdot\beta_-) = LL(\beta)-1$ and still
$SL(\beta\cdot\beta_-)\geqslant 0$. Intuitively, every braid can be 
``relaxed'' into another one with less high twisting.

Our construction of such a braid~$\beta_-$ will again be in the form
of a dance of the punctures of~$D_\beta$, where in a first step each
puncture performs a vertical movement until no two punctures lie at
the same height, in a second step the punctures are squashed onto the
imaginary axis by horizontal movements, and in a third step the punctures
are brought back to the horizontal axis by a $90^{\mathrm o}$
counterclockwise rotation of the vertical axis.

The only step that needs to be defined in a more detailed manner is the
first one. In order to do so, we first classify the kinds
of segments which one can see between successive points with vertical 
tangency in the diagram~$D_\beta$. Firstly, there are those segments which 
start in a right turn and end in a left turn (see 
Figure~\ref{F:UntangleKey}(a)); these correspond to local maxima of 
the function~$\tau_\alpha$. Secondly, there
are those segments that start in a left and end in a right turn (see 
Figure~\ref{F:UntangleKey}(b)); these correspond to local minima of the 
function~$\tau_\alpha$. Thirdly, there are those that start and end in a
right turn, and, fourthly, those that start and end in a left turn 
(see Figure~\ref{F:UntangleKey}(c) and (d)).
The key to the construction of~$\beta_-$ is the following lemma, which
is also illustrated in Figure~\ref{F:UntangleKey}.

\begin{figure}[htb]
\centerline{\input{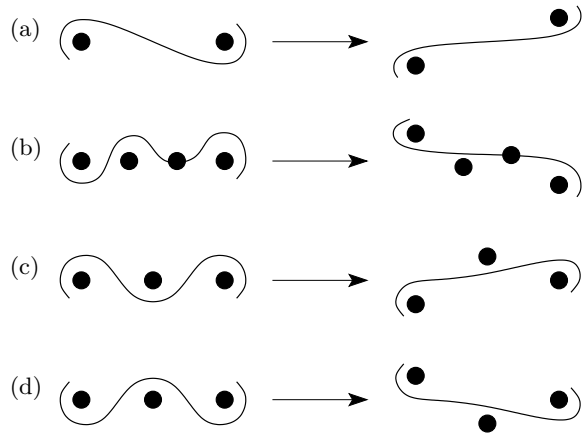}}
\caption{Four types of segments of curve between points of vertical
tangency, and how to deform them by vertical movements of the punctures.}
\label{F:UntangleKey}
\end{figure}

\begin{sublem}\label{L:UntangleKey}
There exists a diffeotopy of the disk which moves every point only
vertically up or down and which deforms the diagram~$D_\beta$ into a
diagram~$D'$ such that:\vspace{-2mm}
\begin{enumerate}
\item Arcs of the first and second type in~$D_\beta$ give rise to arcs
in~$D'$ which do not have any horizontal tangencies
\item Arcs of the third type in~$D_\beta$ give rise to arcs in~$D'$ 
which have precisely one maximum and no minimum in the vertical 
direction (and hence exactly one horizontal tangency)
\item Arcs of the fourth type in~$D_\beta$ give rise to arcs in~$D'$ 
which have precisely one minimum and no maximum in the vertical direction.
\end{enumerate} 
\end{sublem}

\begin{proof}[Proof of Sublemma~\ref{L:UntangleKey}]
We define a relation on the set of punctures in the diagram~$D_\beta$
by saying that a puncture~$p_1$ is \emph{below} a puncture~$p_2$ if 
there is some segment of~$D_\beta$ which contains no vertical tangencies
and such that~$p_1$ lies below the segment, or possibly on it, 
and such that~$p_2$ lies above, or possibly on, the segment -- but at 
most one of the two punctures is supposed to lie on the segment. 
We observe that 
this relation is a partial order. Let us choose any extension of
this partial order to a total order. Now the desired diffeotopy can 
be obtained by sliding the punctures of~$D_\beta$ up or down so that 
their vertical order corresponds to the total order just defined.
\end{proof}

This completes our description of the first step of the puncture dance,
and hence the definition of the braid~$\beta_-$. 

Now we have to prove that the action of~$\beta_-$ simplifies the
curve diagram as claimed. Let us look at an arc of~$D_\beta$ between
two successive vertical tangencies which carries the label~$LL(\beta)$,
i.e.~the largest label that occurs. Such an arc is of the first type, 
in the above
classification. Let us suppose that $\alpha\co I\to D^2$ is a 
parametrization of the curve diagram~$D_\beta$, and that~$I'$ is
a subinterval of~$I$ such that $\alpha|_{I'}$ parametrizes the arc
under consideration. On the interval~$I'$, the function $T_{\alpha}$
takes values in the range $]LL(\beta)-\frac{1}{2},LL(\beta)+\frac{1}{2}]$.
The first step of our particle dance, however, deforms~$\alpha$ into
a parametrized arc~$\alpha'$ such that~$T_{\alpha'}$ only takes values
in the range~$]LL(\beta)-\frac{1}{2},LL(\beta)[$.

Now the second step deforms the arc~$\alpha'$ into an arc~$\alpha''$.
Since this step only moves points horizontally, it never creates or
destroys a point of horizontal tangency in the diagram, so on the
interval~$I'$ the function~$T_{\alpha''}$ only takes values
in the range~$]LL(\beta)-1,LL(\beta)[$. Finally, the third step acts
as a $90^{\mathrm o}$ counterclockwise rotation on the arc, 
so the values of~$T_{\alpha'''}$ lie in the 
interval~$]LL(\beta)-\frac{3}{2},LL(\beta)-\frac{1}{2}[$.
Therefore the segment parametrized by $\alpha'''|_{I'}$ is
part of a segment in the curve diagram of~$D_{\beta \beta_+}$
which is labelled~$LL(\beta)-1$.

Next we claim that an arc of~$D_\beta$ which is labelled by~$SL(\beta)$,
i.e.~the minimal possible label, gives rise in~$D_{\beta \beta_+}$
to part of an arc which is still labelled $SL(\beta)$, and in particular
still positive. This proof is similar to the argument just presented,
and it is left to the reader.
\end{proof}


\section{Are maximally labelled arcs rare?}

There is one additional observation to be made about our
labellings of curve diagrams. Not only do the extremal labels determine
the Garside length of the braid, but moreover it seems that these extremal labels occur only very rarely in the diagram.

\begin{example}
In the curve diagram of the braid in Figure~\ref{F:LabelEx}, there is only one arc segment carrying the maximal label~$2$, and only one carrying the minimal label~$-2$.
\end{example}

\begin{example}
For $\beta_1=\sigma_3^{-1}\sigma_1^{-1}\sigma_2^{-1}\sigma_1^{-1}\sigma_5^{-1}\sigma_4^{-1}\sigma_5^{-1}.\sigma_3 \sigma_2 \sigma_4 \sigma_3 \sigma_1 \sigma_2 \sigma_5 \sigma_4 \sigma_3\in B_6$ we have $\sup(\beta_1)=1$ and $\inf(\beta_1)=-1$.
The curve diagram of $\beta$ is shown in Figure~\ref{F:counterex}(a).
In this diagram, there are two parallel arcs carrying the maximal label~1.
In a similar manner, one can construct curve diagrams of braids in $B_{2k+2}$, with $k$ parallel arcs labelled $1$, where all labels belong to $\{-1,0,1\}$.
\end{example}

\begin{figure}[htb]
\centerline{\input{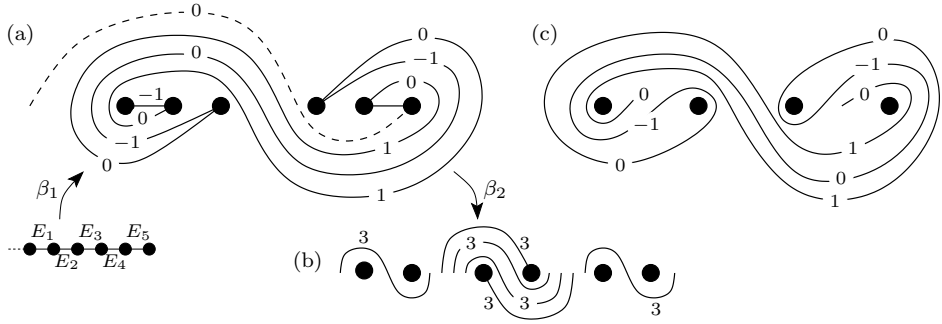}}
\caption{(a) The curve diagram of the braid $\beta_1=\sigma_3^{-1}\sigma_1^{-1}\sigma_2^{-1}\sigma_1^{-1}\sigma_5^{-1}\sigma_4^{-1}\sigma_5^{-1}.\sigma_3 \sigma_2 \sigma_4 \sigma_3 \sigma_1 \sigma_2 \sigma_5 \sigma_4 \sigma_3$. There are two parallel arcs labelled 1  (b) The maximally labelled arcs of the curve diagram of the braid $\beta=\beta_1\beta_2$, with $\beta_2=\sigma_2^3\sigma_4^3 \sigma_1^2 \sigma_3^2 \sigma_5^2$. (c) An arc with~$k$ parallel maximally labelled subarcs in~$D_{2k}$, here for $k=2$.} 
\label{F:counterex}
\end{figure}

We did not manage to construct any worse example, i.e., any example of a curve diagram where there are~$k$ parallel arc segments all carrying the maximal label, in a disk with fewer than $2k+2$ punctures.

\begin{example}
In the curve diagram of the 6-strand braid $\beta=\beta_1\beta_2$, with~$\beta_1$ as in the previous example and $\beta_2=\sigma_2^3 \sigma_4^3 \sigma_1^2 \sigma_3^2 \sigma_5^2$, there are six subarcs carrying the maximal label 3. The curve diagram is quite complicated, so only the maximally labelled arcs are shown in Figure~\ref{F:counterex}(b).
\end{example}

We did not manage to construct any braid on $n$ strands whose curve diagram has more than $n$ maximally labelled arcs.

\begin{example}
It is possible to construct an arc in $D_{2k}$ with $k$ parallel, maximally labelled subarcs, for any integer $k$ with $k\geqslant 2$ -- see Figure~\ref{F:counterex}(c) for an example with $k=2$. We did not manage to construct an arc with~$k$ parallel, maximally labelled subarcs in a disk with fewer than~$2k$ punctures.

\end{example}

In order to make precise the assertion that maximally labelled arcs are rare, let us call ``the $i$th arc of the curve diagram of~$\beta$'' the image under the action of $\beta$ of the straight line segment $E_i$ between the $i$th and $i+1$st puncture.


\begin{conj}
For any braid $\beta\in B_n$ and any $i\in\{1,\ldots,n-1\}$, the $i$th arc of the curve diagram, equipped with the winding number labelling, contains no two parallel arc segments which both carry the maximal label.
\end{conj}

For instance, this conjecture asserts that the arc in Figure~\ref{F:counterex}(c) cannot occur in a curve diagram with maximal label~1. This conjecture, if true, would place a bound depending only on~$n$ on the number of arcs carrying the maximal winding number labelling.

\section{Outlook}\label{S:Outlook}

This paper was motivated by the question ``What do quasi-geodesics in
braid groups really look like?''. More precisely, it is an experimental 
observation that any reasonable way of untangling the curve diagram of
a braid~$\beta$ yields a quasi-geodesic representative 
of~$\beta$ \cite{WiestRelax}. However, the question what precisely is ``reasonable'' has turned out to be very difficult. 

An essential portion of the answer seems to be provided by the insight of
Masur and Minsky \cite{MM2} that for any braid (or mapping class) there
are only finitely many subsurfaces whose interior is tangled by the
action of the braid, and by the Masur-Minsky-Rafi distance 
formula~\cite{RafiDistFormula}. It is, however, not clear how
these results can be used in practice to prove, for instance, that
the Bressaud normal form~\cite{Bressaud,DDRW2} or the 
transmission-relaxation normal form~\cite{DynnWie} are quasi-geodesic.
It would be very useful to have a more practical, or algorithmic, 
version of these ideas.

Returning to ``reasonable'' ways of untangling, let us 
denote~$\Delta_{i,j}$ the braid in which strands number 
$i,i+1,\ldots,j$ perform a half twist. Let us define the 
$\tau$-length of a Garside-generator 
in the following way: it is its length when written as a word in the
generators $\Delta_{i,j}$ (equivalently, it is the number of half
Dehn twists along round curves in $D_n$ needed to express it). Now,
given a braid $\beta$, put it in right Garside normal form,
and add up the lengths of its factors. We shall call the result the
$\tau$-length of $\beta$.

For a subdisk $D\subseteq D_n$ which is round (contains punctures
number $i,i+1,\ldots,j$ for some $1\leqslant i<j\leqslant n$),
and a braid $\beta$, we define the \emph{tangledness} of
the curve diagram of~$\beta$ in~$D$ in the following way.
The intersection of the curve diagram of~$\beta$ with~$D$ consists
of a (possibly very large) number of arcs, which inherit a labelling
from the winding number labelling of the full diagram.
Suppose that you shift the labelling of each arc inside the subdisk 
so that the two extremities of each arc \vspace{-2mm}
\begin{itemize}
\item either both lie in a segment labelled 0, or
\item one lies in a segment labelled 0 and the other lies in a segment
 labelled~1.\vspace{-2mm}
\end{itemize}
Now for each arc take
$$(\text{the largest label} - \text{the smallest label})
\text{ or }
(\text{the largest label} - 1 - \text{the smallest label})$$
according to the type of the arc. Take the maximum of these quantities
over all arcs. That's the tangledness. 

\begin{conj}\label{C:TauLengthCong}
Suppose that in the curve diagram of a  braid $\beta$ there is a round 
disk whose interior curve diagram has strictly positive tangledness.
Suppose that~$\beta_+$ is a Garside generator or the inverse of a Garside
generator, that it
moves only strands inside the round disk, and that its action reduces 
the tangledness of the diagram inside the round disk. Then $\beta\beta_+$ 
has smaller $\tau$-length than $\beta$.
\end{conj}\vspace{-1mm}

\begin{question} {\rm 
For a pseudo-Anosov braid~$\beta$ we can look at the maximally and
minimally labelled arcs in the curve diagrams of \emph{high powers} 
of~$\beta$, and, by passing to the limit, in a train track or in the 
stable foliation of~$\beta$. 
The obvious question is: what do the positions of these arcs tell us? 
Could they, for instance, be helpful for solving the conjugacy problem?}
\end{question}


\begin{question} {\rm Is there an analogue of the main result 
(Theorem~\ref{T:label=length})
for $Out(F_n)$, or at least a substiantial subgroup of $Out(F_n)$?}
\end{question}

{\bf Acknowledgements } This paper is a branch of the paper~\cite{GMWred},
and I am extremely grateful to Juan Gonz\'alez-Meneses for many helpful 
discussions. I thank Tetsuya Ito who played an important role in finding 
a mistake in an earlier version of Section~3. In its current form, 
Section~3 is joint work with him.


\end{document}